\documentclass[12pt]{amsart}
\pdfoutput=1
\usepackage{amsmath, amssymb, amsthm, amsbsy, bbm, amsfonts, color, verbatim, array, floatflt, yfonts}

\usepackage[margin=1.2in]{geometry}
\usepackage[all]{xy}
\usepackage[english]{babel}
\usepackage[shortlabels]{enumitem}
\usepackage{ stmaryrd } 

\usepackage{graphicx, color, overpic}
\usepackage[usenames,dvipsnames,svgnames,table]{xcolor}
\usepackage[pdfpagelabels]{hyperref}
\usepackage{cleveref}
\usepackage{mathtools}

\usepackage{multirow}
\usepackage{easybmat}

\theoremstyle{plain}
\newtheorem{thm}{Theorem}[section]
\newtheorem*{thm*}{Theorem}
\newtheorem{prop}[thm]{Proposition}
\newtheorem*{prop*}{Proposition}
\newtheorem{lemma}[thm]{Lemma}
\newtheorem*{lemma*}{Lemma}
\newtheorem{corollary}[thm]{Corollary}

\theoremstyle{definition}
\newtheorem{definition}[thm]{Definition}
\newtheorem{example}[thm]{Example}

\newtheorem{rmk}[thm]{Remark}
\newtheorem{algorithm}[thm]{Algorithm}

\newcommand{\R}{\mathbb{R}}
\newcommand{\Z}{\mathbb{Z}}

\newcommand{\E}{\mathbb{E}}
\newcommand{\bS}{\mathbb{S}}

\newcommand{\Id}{\operatorname{I}} 
\newcommand{\aG}{G} 
\newcommand{\aH}{H} 
\newcommand{\sH}{H_0} 
\newcommand{\On}{\operatorname{O}(n)} 
\newcommand{\Otwo}{\operatorname{O}(2)} 

\newcommand{\Isom}{\operatorname{Isom}} 
\newcommand{\Mov}{\textsc{Mov}} 
\newcommand{\Mod}{\textsc{Mod}} 
\newcommand{\ModH}{\textsc{Mod}_\aH} 
\newcommand{\Fix}{\textsc{Fix}} 
 
\newcommand{\Cent}{\operatorname{C}} 

\newcommand{\Range}{\operatorname{Im}}
\newcommand{\Ker}{\operatorname{Ker}}

\newcommand{\hconj}{[h]_\aH}

\newcommand{\conj}[1]{[#1]} 

\newcommand{\coconj}{\operatorname{C}} 
\newcommand{\coconjH}[2]{\operatorname{C}_\aH(#1,#2)} 
\newcommand{\scoconjH}[2]{\operatorname{C}_{\sH}(#1,#2)} 
\newcommand{\Comp}{\operatorname{Comp}} 
\newcommand{\CompH}{\operatorname{Comp}_\aH} 
\newcommand{\CompG}{\operatorname{Comp}_\aG} 
\newcommand{\BaseH}{\operatorname{Base}_\aH} 
\newcommand{\BaseG}{\operatorname{Base}_\aG} 
 
\numberwithin{equation}{subsection}

\definecolor{amethyst}{rgb}{0.6, 0.4, 0.8}
\definecolor{kellygreen}{rgb}{0.3, 0.73, 0.09}
\definecolor{americanrose}{rgb}{1.0, 0.01, 0.24}

\begin{document}

\hypersetup{pdfauthor={Milicevic, Schwer, Thomas},pdftitle={The geometry of conjugation in Euclidean isometry groups}}

\title[Conjugation in Euclidean Isometry Groups]{The Geometry of Conjugation in \\ Euclidean Isometry Groups}

\author{Elizabeth Mili\'{c}evi\'{c}}
\address{Elizabeth Mili\'{c}evi\'{c}, Department of Mathematics \& Statistics, Haverford College, 370 Lancaster Avenue, Haverford, PA, USA}
\email{emilicevic@haverford.edu}

\author{Petra Schwer}
\address{Petra Schwer, Heidelberg University, Department of Mathematics and Computer Science, Im Neuenheimer Feld 205, 69120 Heidelberg, Germany}
\email{schwer@uni-heidelberg.de}

\author{Anne Thomas}
\address{Anne Thomas, School of Mathematics \& Statistics, Carslaw Building F07,  University of Sydney NSW 2006, Australia}
\email{anne.thomas@sydney.edu.au}

\thanks{EM was supported by NSF Grant DMS 2202017. PS was supported by the DFG Project SCHW 1550/4-1. This research was also supported in part by ARC Grant DP180102437.}



\begin{abstract}
We describe the geometry of conjugation within any split subgroup $\aH$ of the full isometry group $\aG$ of $n$-dimensional Euclidean space.  We prove that for any $h \in \aH$, the conjugacy class $\hconj$ of $h$ is described geometrically by the move-set of its linearization, while the set of elements conjugating $h$ to a given $h'\in\hconj$ is described by the the fix-set of the linearization of $h'$. Examples include all affine Coxeter groups, certain crystallographic groups, and the group $\aG$ itself. 
 
\end{abstract}

\maketitle




\section{Introduction}\label{sec:Intro}

Group theory has a long history of studying conjugacy classes and the conjugation problem.
It turns out that conjugacy classes in the full isometry group $\aG$ of $n$-dimensional Euclidean space $\E^n$, as well as in all split subgroups $\aH$ of $\aG$, have a simple and beautiful geometric description.

\begin{thm}\label{thm:intro}
Let $H$ be any subgroup of the full isometry group of $\E^n$ which admits a semi-direct product splitting into an $O(n)$-part and a translation part. Then:
\begin{enumerate}
	\item\label{intro:conj} The conjugacy class of any $h = t^\lambda w \in H$ is determined by the move-set of its linearization $w$. As a set, the class $\conj{h}$ is obtained by first translating $h$ by all elements of $\textsc{Mod}_{H} (w) \subseteq \Mov(w)$, and then conjugating the so-obtained collection $t^{\textsc{Mod}_{H} (w)}h$ by all elements of $O(n)\cap H$. 
	\item\label{intro:coconj} The coconjugation set $\coconj(h,h') = \{ k \in H \mid khk^{-1} = h'\}$ for any $h, h' \in H$  has a closed-form description in terms of data involving the move-set $\Mov(w')$ of the linearization $w'$ of $h'$, and its shape is described by translates of the orthogonal space $\Fix(w')$. 
\end{enumerate}
\end{thm}

\noindent Precise statements of \Cref{thm:intro} are given in \Cref{thm:ConjClass} and \ref{thm:coconj}. 
An important special case, which provided our original motivation, is when $\aH$ is an affine Coxeter group; we give additional results in this setting in our companion paper~\cite{MST5}. 

Our point of view on Euclidean isometries is influenced by works of McCammond and his coauthors~\cite{BradyMcCammond, McCammondSulway, McCammond, LMPS}. We were surprised to discover that our approach to conjugation seems to be new, even for the full isometry group $\aG$.

\subsection{Structure of the paper} Section~\ref{sec:intro-mainresults} contains the formal statements of our main results, illustrated by many examples. Section~\ref{sec:Conj} concerns conjugacy classes and contains the proof of Theorem~\ref{thm:ConjClass}. 
We consider the coconjugation problem in Section~\ref{sec:CoConj}, where we prove \Cref{thm:coconj}. An algorithm for solving the coconjugation problem, and hence the conjugacy problem,  is contained in Section~\ref{sec:algorithm}.

\subsection{Acknowledgements}
We thank Martin Bridson, Benson Farb, Eamonn O'Brien, and Olga Varghese for helpful conversations. AT and PS are grateful for Haverford's support during their visits to EM in June 2022 and April 2023, respectively. 
PS was partially supported by the DFG grant no.~314838170, GRK 2297 \emph{MathCoRe}.

\section{Results and examples}\label{sec:intro-mainresults}

It is classical that the full isometry group $\aG$ of $\E^n$ splits as a semidirect product $\aG = T \rtimes \On$, where $T \cong \R^n$ is the translation subgroup of~$\aG$ and $\On$ is the group of orthogonal transformations. 
We consider subgroups $\aH$ of $\aG$ which respect this splitting; that is, where $\aH =  T_\aH \rtimes \sH$ for $T_\aH = T \cap \aH$ and $\sH = H \cap \On$. 
For any such $\aH \leq \aG$ and for all $h,h' \in \aH$, we write 
\begin{align*}
\hconj &= \{ khk^{-1} \mid k \in \aH \} \quad \text{ for the \emph{conjugacy class of }} h\in \aH \text{ and }\\
\coconjH{h}{h'} &= \{ k \in \aH \mid khk^{-1} = h' \} \quad \text{ for the \emph{coconjugation set  (from $h$ to $h'$)}}.
\end{align*}	
In particular, $\Cent_\aH(h,h)$ is the centralizer of $h$ in $\aH$, which we also denote by $\Cent_\aH(h)$.

For any $\lambda \in \R^n$, we write $t^\lambda$ for the translation of $\E^n$ by the vector $\lambda$. For any split $\aH \leq \aG$, we define $L_\aH = \{ \lambda \in \R^n \mid t^\lambda \in T_\aH \}$, and observe that $L_\aH$ is naturally a $\Z$-module. Then any~$h \in \aH$ can be expressed uniquely as $h = t^\lambda h_0$, where $\lambda \in L_\aH$ and $h_0 \in \sH$. We call $t^\lambda$ the \emph{translation part} and $h_0$ the \emph{spherical part} of~$h$.  For any $\lambda \in L_\aH$ and $h_0 \in \sH$, we have $h_0 t^\lambda h_0^{-1} = t^{h_0\lambda}$.

\subsection{Conjugation}

Following Brady--McCammond~\cite{BradyMcCammond} (but identifying $\E^n$ with $\R^n$), we recall that the \emph{move-set} and \emph{fix-set} of any isometry $g \in \aG$ are the affine subspaces of $\R^n$ given by, respectively,
\[
\Mov(g) = \{ y \in \R^n \mid gx = x + y \mbox{ for some $x \in \R^n$}\} = (g-\Id)\R^n = \Range(g - \Id)
\]
and
\[
 \Fix(g) = \{ x \in \R^n \mid g x = x \} = \Ker(g - \Id).
\] 
For example, if $r \in G$ is a reflection, then $\Mov(r)$ is the line through the origin orthogonal to the affine hyperplane $\Fix(r)$. 
If $g_0 \in \On$, then $\Mov(g_0)$ and $\Fix(g_0)$ are both linear subspaces, and~$\R^n$ has orthogonal decomposition $\R^n = \Mov(g_0) \oplus \Fix(g_0)$ (see~\cite[Remark 1.8]{LMPS}).  For any $\lambda \in \R^n$ and any $g_0 \in \On$, by~\cite[Proposition 1.21]{LMPS} we have \[\Mov(t^\lambda g_0) = \lambda + \Mov(g_0).\]

We now introduce an $\aH$-adapted version of the move-set, which we call the \emph{mod-set}.

\begin{definition}[Mod-set]
	\label{defn:mod-set} 
	Let $H = T_\aH \rtimes \sH$ be a split group of Euclidean isometries. For any $h \in \aH$, the \emph{mod-set (with respect to $\aH$) of $h$} is defined by:
\[\label{eq:mod} \ModH(h) = (h - \Id) L_\aH. \]
\end{definition}

\noindent Note that since $L_\aH = -L_\aH$, we could equally well have defined $\ModH(h) = (\Id - h)L_\aH$. We show in Lemma~\ref{lem:modShift} that for any $\lambda \in L_\aH$ and any $h_0 \in \sH$,
 \[\ModH(t^\lambda h_0) = \lambda + \ModH(h_0).\] It is immediate from definitions that $\ModH(h) \subseteq \Mov(h)$, and we prove in Lemma~\ref{lem:modMov} that $\ModH(h)$ is contained in the intersection $\Mov(h) \cap L_\aH$.  If $h_0 \in \sH$, then $\Mod(h_0)$ is a submodule of $L_\aH$, and hence a submodule of $\Mov(h_0) \cap L_\aH$. 

Our first main result says that the mod-set is the key to describing the conjugacy class $\hconj$.

\begin{thm}[Closed form of conjugacy classes]
	\label{thm:ConjClass}  
	Let $\aH = T_\aH \rtimes \sH$ be a split group of Euclidean isometries. Let $h = t^\lambda h_0 \in \aH$, where $\lambda \in L_\aH$ and $h_0 \in \sH$.  Then the conjugacy class of $h$ in $\aH$ satisfies
	\begin{equation}\label{eq:conj1}
		\hconj    
		= \bigcup_{u \in \sH} u \left( t^{\ModH(h_0) } h \right) u^{-1} 
		\end{equation}
		and also
	\begin{equation}\label{eq:conj2}
		\hconj    
		= \bigcup_{u \in \sH} t^{u( \lambda + \ModH(h_0) )} u h_0  u^{-1} = \bigcup_{u \in \sH} t^{u\ModH(h) } u h_0  u^{-1}.
		\end{equation}
\end{thm}

\noindent In words, the two equalities of \Cref{thm:ConjClass} tell us that $\hconj$ is obtained by, respectively:
\begin{enumerate}
\item first translating $h$ by all elements of $\ModH(h_0)$, and then conjugating the so-obtained collection $t^{\ModH(h_0)}h$ 
by all elements of $\sH$; or
\item for each $u \in H_0$, translating the $u$-conjugate of the spherical part $h_0$ of $h$ by the set $t^{u( \lambda + \ModH(h_0) )}=t^{u\ModH(h)}$.
\end{enumerate} 
This second description in particular implies that the conjugacy class of every element $h=t^\lambda h_0\in\aH$ is determined by the conjugacy class of $h_0$ in $H_0$, together with the collection of images of mod-sets $u( \lambda + \ModH(h_0)) = u\ModH(h)$, as $u$ runs through the elements of $\sH$. 

As an easy consequence of Theorem~\ref{thm:ConjClass} and the fact that $\ModH(h) \subseteq \Mov(h) \cap L_\aH$, we obtain that (the translation parts of) conjugacy classes lie along move-sets:

\begin{corollary}[Conjugacy classes and move-sets]\label{cor:conjMov} Let $\aH = T_\aH \rtimes \sH$ be a split group of Euclidean isometries. Let $h = t^\lambda h_0 \in \aH$, where $\lambda \in L_\aH$ and $h_0 \in \sH$.   Then,
\begin{equation}\label{eq:conjMov1}
		\hconj    
		\subseteq \bigcup_{u \in \sH} u \left( t^{\Mov(h_0) \cap L_\aH} h \right) u^{-1}
		\end{equation}
and also
\begin{equation}\label{eq:conjMov2}
		\hconj  
		\subseteq \bigcup_{u \in \sH} t^{u(\Mov(h) \cap L_\aH)} u h_0  u^{-1}.
		\end{equation}
\end{corollary}

\noindent Our second main definition is motivated by these containments.

\begin{definition}[Filling]\label{defn:filling} Let $\aH = T_\aH \rtimes \sH$ be a split group of Euclidean isometries. We say that $h \in \aH$ \emph{fills its move-set}, or that \emph{filling occurs for $h$}, if
\[
\ModH(h) = \Mov(h) \cap L_\aH.
\]
\end{definition}

\noindent We prove in Proposition~\ref{prop:conjFilling} that filling occurs (for both $h$ and $h_0$) if and only if the containments in Corollary~\ref{cor:conjMov} are equalities.

\begin{example}\label{eg:cmm} Let $\aH$ be the wallpaper group $\mathbf{cmm}$, denoted $2^*22$ in orbifold notation. Then $\aH$ is split, $\sH$ is the Klein four group generated by two commuting reflections, say $s_1$ and $s_2$, and $\aH$ is generated by $s_1$, $s_2$, and a $180^\circ$ rotation, say $\rho$, about a point not on any reflection axis. The group $\aH$ induces the tesselation of $\E^2$ by triangles depicted in Figures~\ref{fig:cmm_conj_reflection_cropped} and~\ref{fig:cmm_cuspidal}, and $L_H$ is the lattice of heavy dots in these figures. There is a natural bijection between the elements of $\aH$ and the tiles in these tesselations, and we identify each element of $H$ with its corresponding tile. A few tiles are labeled in Figure~\ref{fig:cmm_conj_reflection_cropped}.

  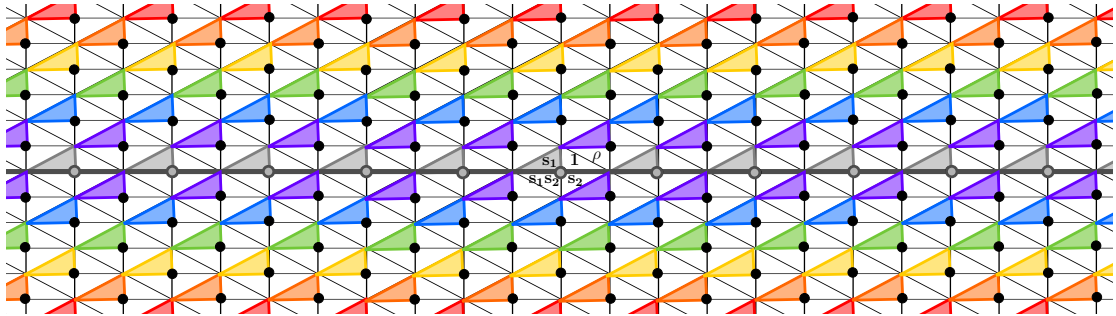
\begin{figure}[h]
		\resizebox{0.95\textwidth}{!}
		{\begin{overpic}{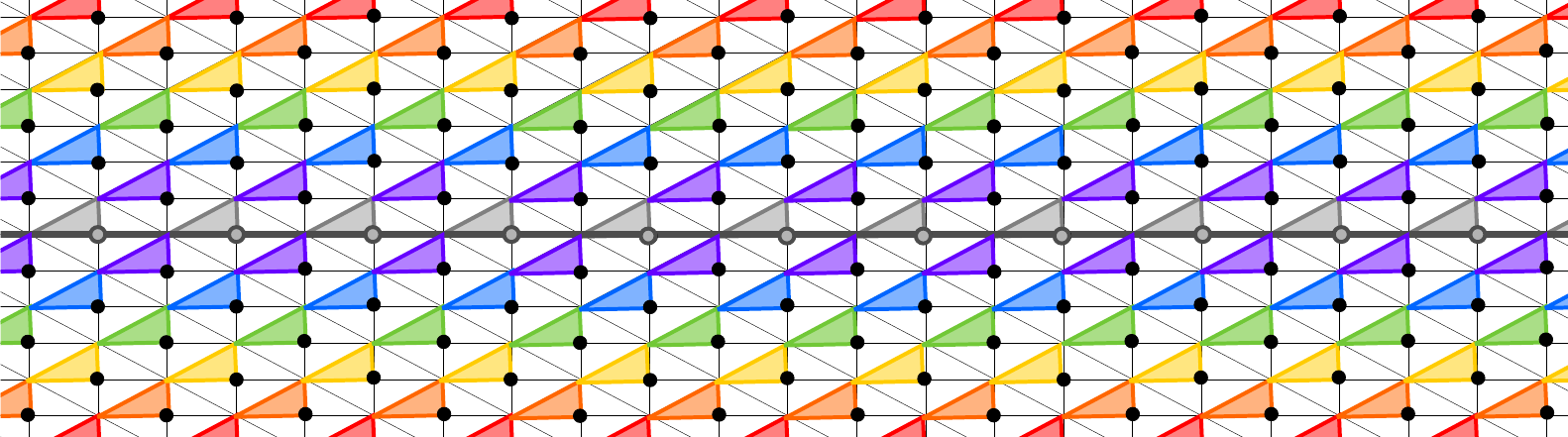}
		\put(51,13.5){$\mathbf{1}$}
		\put(48.5,13.5){$\mathbf{s_1}$}
		\put(50.8,11.8){$\mathbf{s_2}$}
		\put(53,14){$\mathbf{\rho}$}
		\put(47.3,11.8){$\mathbf{s_1 s_2}$}	
		\end{overpic}}
		\vspace{-3mm}
		\caption{\footnotesize{Conjugacy classes $[t^\lambda s_1]_\aH$ in the wallpaper group $\mathbf{cmm}$.}}	
		\label{fig:cmm_conj_reflection_cropped}	
\end{figure}

The conjugacy classes in $\aH$ are as follows.  In Figure~\ref{fig:cmm_conj_reflection_cropped}, each set of tiles of the same color is a conjugacy class $[t^\lambda s_1]_\aH$. The mod-set $\ModH(s_1) \subseteq L_\aH$ is the set of large gray dots along the horizontal axis, and the move-set $\Mov(s_1)$ is this horizontal axis. The horizontal lines in the figure are the sets $\Mov(t^\lambda s_1)$. If $\lambda \in \Mov(s_1)$ then $\Mov(t^\lambda s_1) = \Mov(s_1)$ is $\sH$-invariant, and the conjugacy class $[t^\lambda s_1]_\aH$ is the set of gray triangles along the horizontal axis. For $\lambda \not \in \Mov(s_1)$, the line $s_2 \Mov(t^\lambda s_1) = s_2 s_1 \Mov(t^\lambda s_1)$ is distinct from $\Mov(t^\lambda s_1)$, and so $[t^\lambda s_1]_\aH$ is a pair of horizontal ``lines" of triangles (of the same color). The description of the conjugacy classes $[t^\lambda s_2]_\aH$ is similar, just involving vertical ``lines" of triangles instead. Note that $s_1$ and $s_2$ both fill their move-sets.

Figure~\ref{fig:cmm_cuspidal} depicts the conjugacy classes $[t^\lambda s_1 s_2]_\aH$. Since $s_1 s_2 = -\Id$, the mod-set $\ModH(s_1s_2)$ equals $2L_\aH$, shown by large gray dots, and $\Mov(s_1 s_2) = \R^2$. Thus $s_1 s_2$ does not fill its move-set. The left of Figure~\ref{fig:cmm_cuspidal} shows the conjugacy class $[t^\lambda s_1 s_2]_\aH$ for any $\lambda \in \ModH(s_1s_2)$, while the center and right show the cases $\lambda \in L_\aH\setminus \ModH(s_1s_2)$. In the center, the set $\ModH(t^\lambda s_1 s_2) = \lambda + \ModH(s_1s_2)$ is $\sH$-invariant, while on the right, the $\sH$-orbit of $\ModH(t^\lambda s_1s_2)$ has $2$ elements, corresponding to the dark and light pink triangles. 

Finally, for any $\lambda \in L_\aH$, the class $[t^\lambda]_\aH$ is just the finite set $\{ t^{h_0\lambda} \mid h_0 \in \sH \}$.
\end{example}

\begin{figure}[htb]
	\begin{minipage}{0.3\textwidth}
		\centering
		\includegraphics[width=0.9\textwidth]{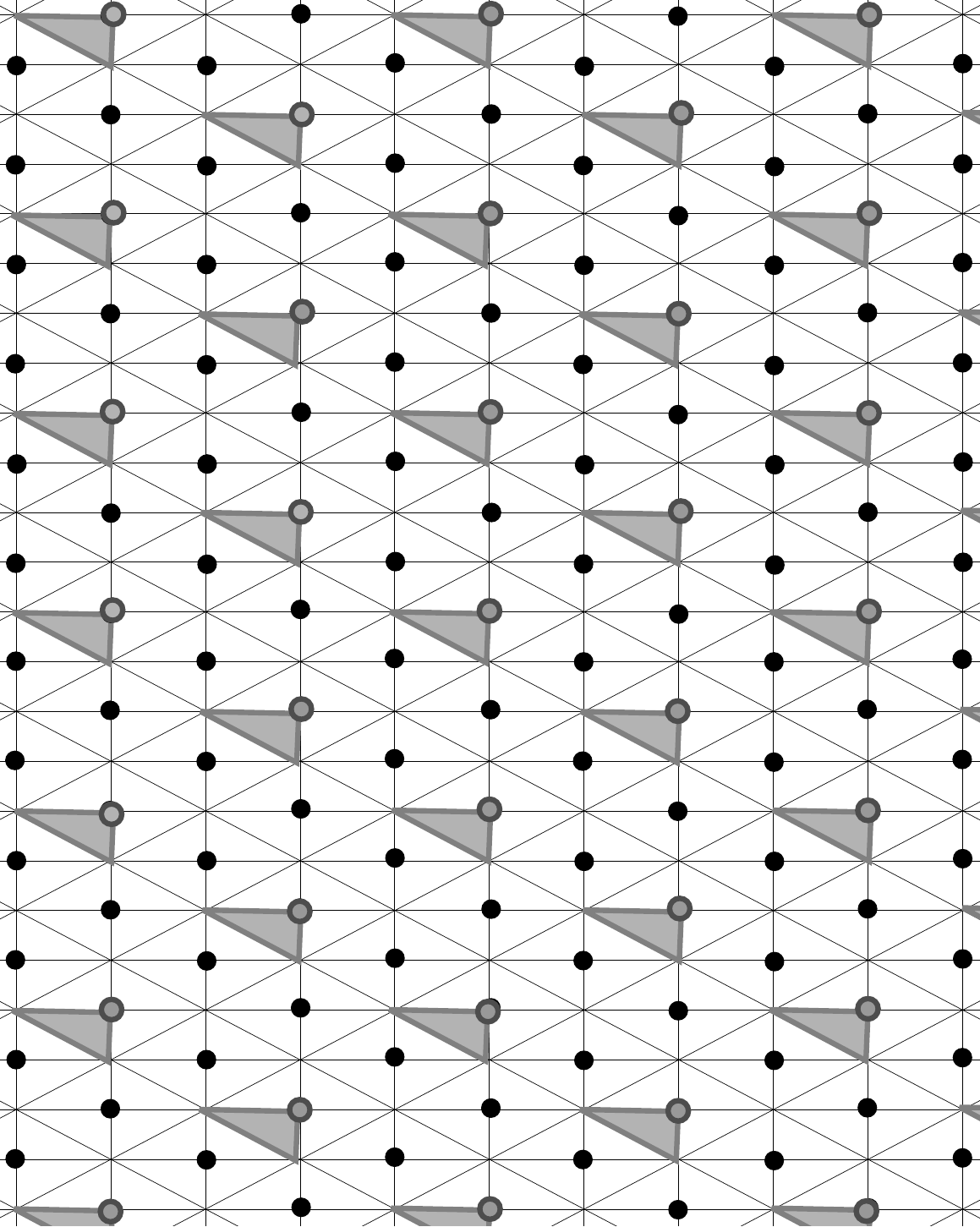}
	\end{minipage}
	\begin{minipage}{0.3\textwidth}
		\centering
		\includegraphics[width=0.9\textwidth]{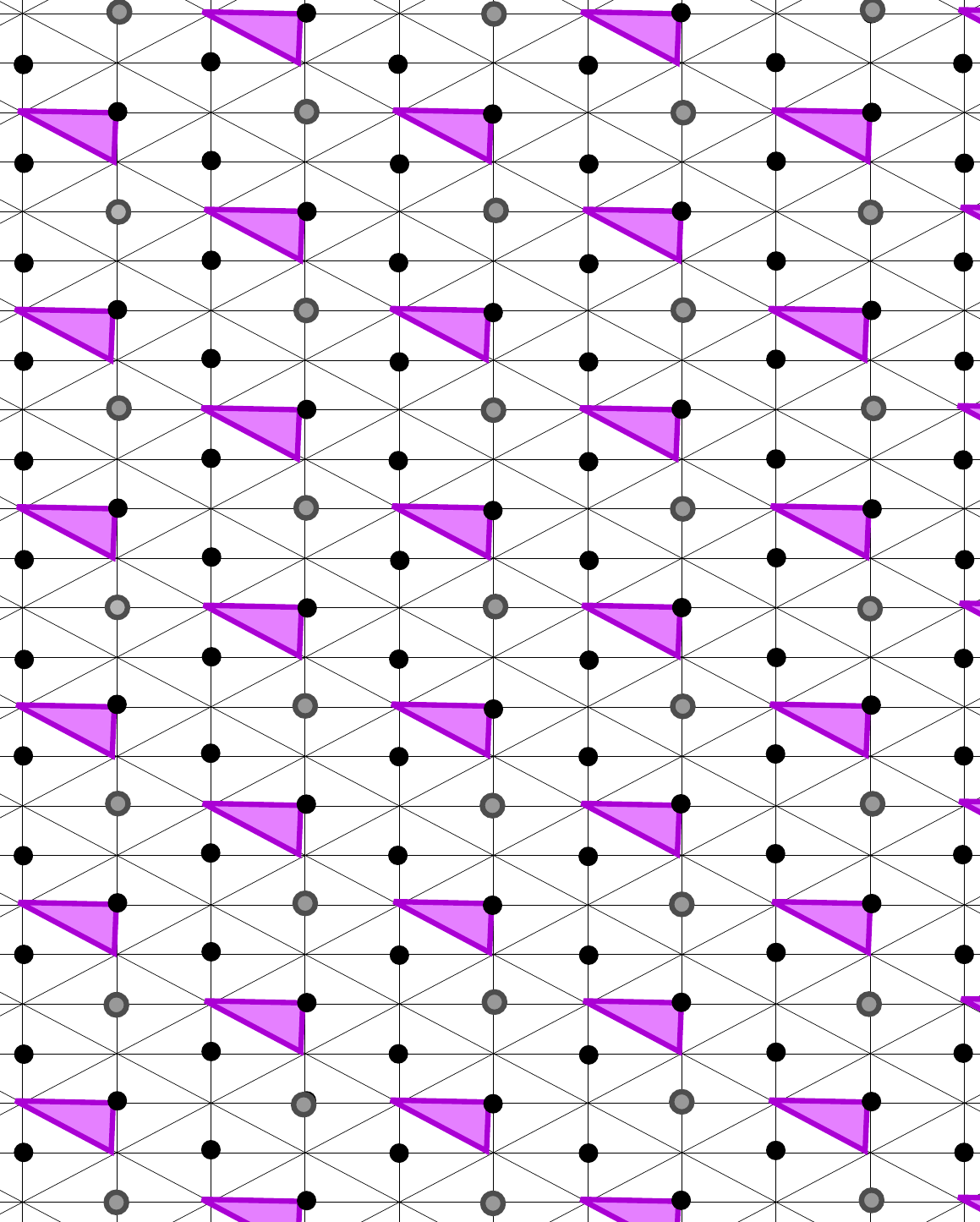}
	\end{minipage}
	\begin{minipage}{0.3\textwidth}
		\centering
		\includegraphics[width=0.9\textwidth]{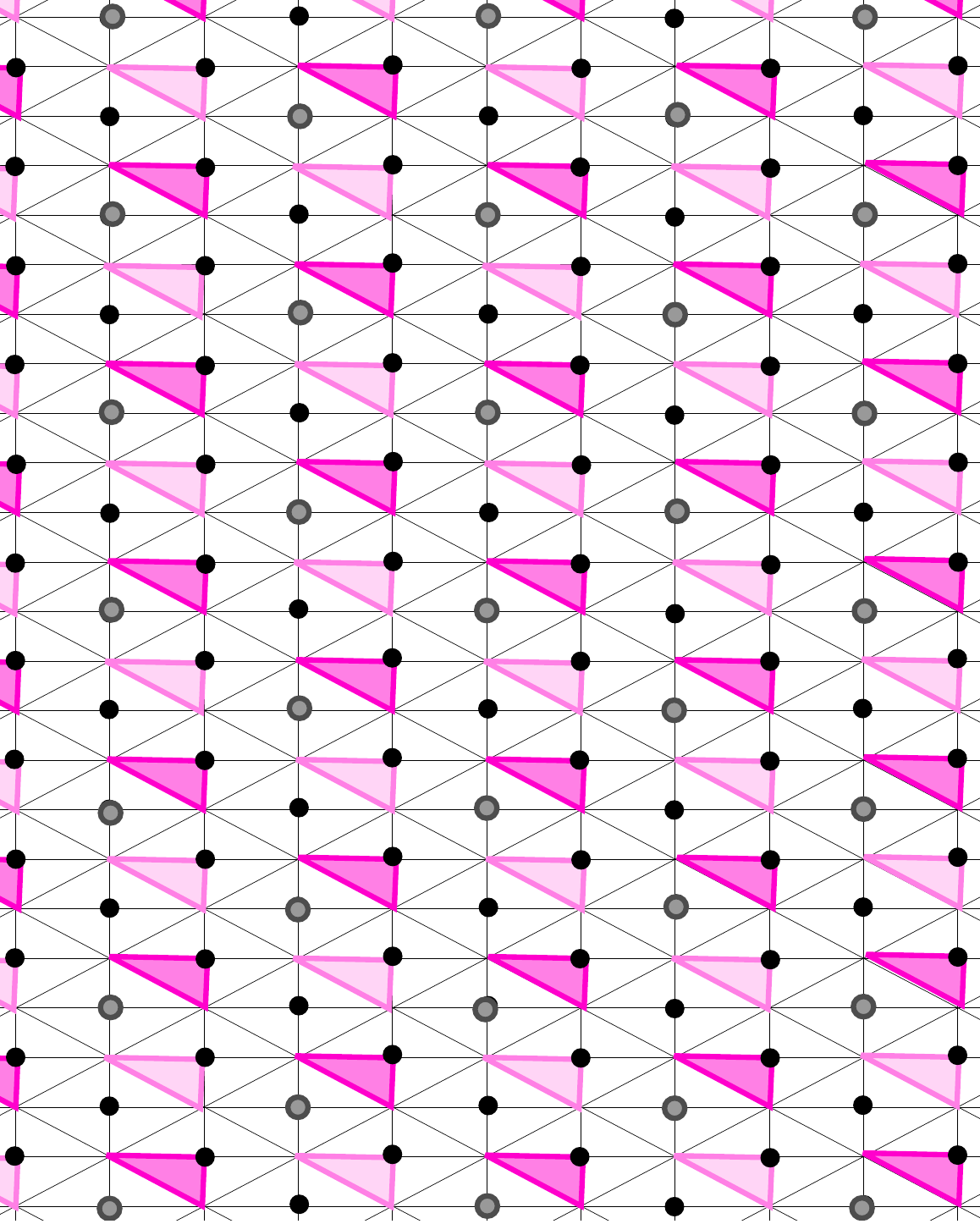}
	\end{minipage}
	\caption{\footnotesize{Conjugacy classes $[t^\lambda s_1 s_2]_\aH$ in the wallpaper group $\mathbf{cmm}$.}}
	\label{fig:cmm_cuspidal}
\end{figure}

\begin{rmk}\label{rmk:spherical}
We point out the importance of the subgroup $\sH$ in our descriptions of conjugacy classes. For example, the spherical part of $h \in \aH$ completely determines the mod-set appearing in the first equality of Theorem~\ref{thm:ConjClass}, and in both equalities in its statement we are conjugating just by elements of $\sH$.  This importance of $\sH$ is perhaps unsurprising if we consider the boundary at infinity $\partial\E^n\cong\bS^{n-1}$, that is, the set of all parallelism classes of rays in~$\E^n$.  The group~$\aH$ acts on $\partial \E^n$ with kernel its translation subgroup $T_\aH$, and so all we are seeing at infinity is the induced action of~$\sH$. 
\end{rmk}

\begin{rmk}\label{rmk:cryst}  An important class of subgroups of $G$ are the \emph{crystallographic} groups. These can be defined as the discrete subgroups of $\aG$ which act cocompactly on $\E^n$; a  crystallographic group is sometimes called \emph{symmorphic} if it splits. Thus a split subgroup $\aH = T_\aH \rtimes \sH$ of $\aG$ is crystallographic exactly when $L_\aH$ is a lattice in $\R^n$ and $\sH$ is finite.  The 17 wallpaper groups, 13 of which split, are the $2$-dimensional crystallographic groups.  Any crystallographic group $\aH \leq \aG$ is of finite index in a split crystallographic $\aH' \leq \aG$ (see, for instance,~\cite[p. 316]{Ratcliffe}), and then obviously every $\aH$-conjugacy class is contained in some $\aH'$-conjugacy class.  Hence Theorem~\ref{thm:ConjClass} provides, up to finite index, a description of all conjugacy classes in all crystallographic groups.
\end{rmk}

\begin{rmk}\label{rmk:affineCoxeter}
 The \emph{affine Coxeter groups}, all of which split, are the crystallographic groups which are generated by reflections in the faces of a convex polytope (compare~\cite[Definition 6.4.4]{Davis}). As explained in~\cite[Appendix B]{MST5}, for $n = 2,3$ every $n$-dimensional crystallographic group is finite index in some affine Coxeter group, but for $n \geq 4$ there are examples of crystallographic groups $\aH \leq \aG = \Isom(\E^n)$ which are not contained in any affine Coxeter subgroup of $\aG$.
\end{rmk}

\begin{example}\label{eg:introG} 
Let $\aH = \aG = \Isom(\E^2)$ be the full isometry group of the Euclidean plane.  Let $\lambda \in \R^2$ be nonzero, let $r \in \Otwo$ be the unique linear reflection which fixes~$\lambda$, and let~$g$ be the glide-reflection $g= t^\lambda r$.  Then the conjugacy class $[g]_\aG$ is the disjoint union of all lines which are tangent to the circle of radius $\| \lambda \|$, as depicted in Figure~\ref{fig:introG}.  More precisely, if~$\ell$ is a line tangent to this circle, then the point $p$ of~$\ell$ corresponds to the element $t^p r_\ell$ of~$[g]_\aG$, where $r_\ell \in \Otwo$ is the unique linear reflection preserving $\ell$. 

\begin{figure}[htb]
	\begin{center}
		\includegraphics[width=0.35\textwidth]{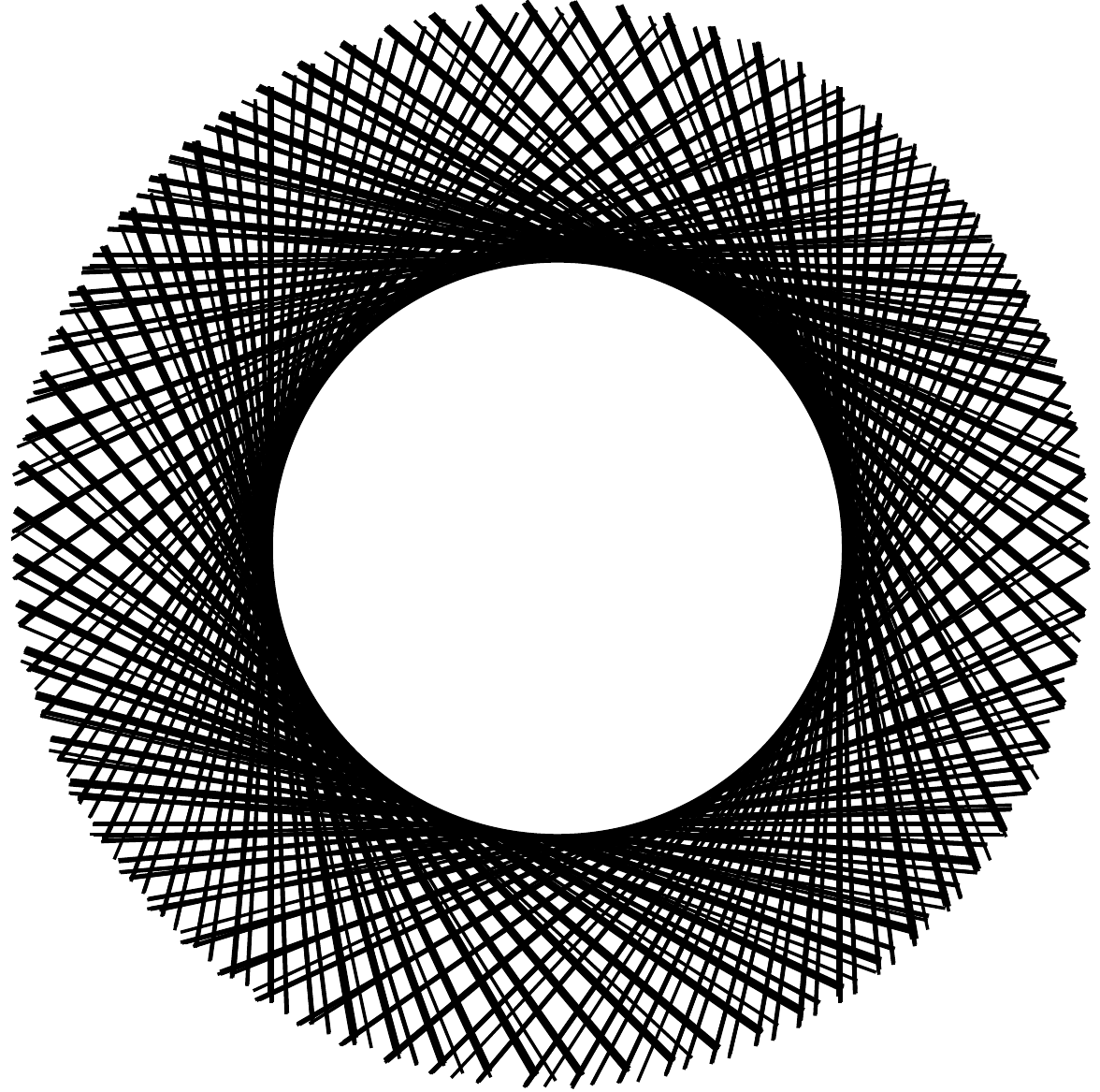}
	\end{center}
	\caption{\footnotesize{The conjugacy class of a glide-reflection in $\Isom(\E^2)$.}}
	\label{fig:introG}
\end{figure} 

Note that each same-color pair of ``lines" of tiles in Figure~\ref{fig:cmm_conj_reflection_cropped} can be viewed as a ``discrete shadow" of a pair of actual lines in Figure~\ref{fig:introG}.

Now take any $\lambda \in \R^2$ and let $g_0 = -\Id$. Then $g = t^\lambda g_0$ is the rotation by $180^\circ$ about the point~$\frac{1}{2}\lambda$, and all such rotations are conjugate in $\aG$. Hence the three distinct conjugacy classes of $\mathbf{cmm}$ seen in Figure~\ref{fig:cmm_cuspidal} are all discretizations of a single conjugacy class in $\aG$.

For translations, given any $\lambda \in \R^2$ we can identify the conjugacy class $[t^\lambda]_\aG = \{ t^{u\lambda} : u \in \Otwo \}$ with the circle of radius $\| \lambda \|$. Thus the finitely many conjugates of any translation in  $\mathbf{cmm}$ are again just a discrete glimpse of its full conjugacy class in $\aG$.

We can generalize the previous paragraph as follows: the conjugacy classes of translations in $\aG = \Isom(\E^n)$ partition $\R^n$ into the set of spheres centered at the origin, and for any split crystallographic subgroup $\aH \leq \aG$, the conjugacy classes of translations in $\aH$ pick out finitely many points on each of a discrete set of these spheres.
\end{example}

\begin{rmk}
The full isometry group $\aG$ of $\E^n$ is a Lie group, and the conjugation action of $\aG$ on itself is smooth.  It follows that conjugacy classes in $\aG$, which are the orbits of this action, are immersed submanifolds of $\aG$.  We note that the conjugation action of $\aG$ on itself is not proper, since the entire non-compact group $\aG$ fixes its identity element.  There is considerable literature on conjugacy classes in compact Lie groups, such as $\On$, where the conjugation action is proper and there is a classical relationship to representation theory (see, for example, \cite[Chapter IV]{BroeckerDieck}). However, we do not know of any work on conjugacy classes in $\aG$ itself from a Lie-theoretic point of view. Our proofs use only the algebraic structure of $\aG$, as a semidirect product.
\end{rmk}

We now continue the description of our main results.  
The equalities in \Cref{thm:ConjClass} suggest that a conjugacy class naturally decomposes into a family of subsets. We explore this phenomenon by considering the \emph{components} of the conjugacy class $\hconj$; that is, the subsets of $\hconj$ of the form $u(t^{\ModH(h_0)} h)u^{-1}$ where  $u\in \sH$. 
Write $\CompH(h)$ for the set of components of $\hconj$. By definition, the group $\sH$ acts transitively by conjugation on $\CompH(h)$.

\begin{thm}[Components] 
	\label{thm:Components}  
	Let $\aH = T_\aH \rtimes \sH$ be a split group of Euclidean isometries. Let $h = t^\lambda h_0 \in \aH$, where $\lambda \in L_\aH$ and $h_0 \in \sH$. Then:
	\begin{enumerate}
		\item The conjugation action of the group $T_\aH$ induces a transitive action by translation on the elements of each component of $\hconj$.
		\item Linearization induces a natural surjection from $\CompH(h)$ to $\CompH(h_0)$. 
		\item There is a natural bijection between  $\CompH(h_0)$ and  $[h_0]_{\sH}$.   
	\end{enumerate}
\end{thm}

Part (1) of Theorem~\ref{thm:Components} completes a ``global" picture of the action of $\aH = T_\aH \rtimes \sH$ on itself by conjugation: the elements of $\sH$ permute the components of any conjugacy class, while the translation subgroup $T_\aH$ acts transitively within each component. In Proposition~\ref{prop:equalcomponents}, we determine the stabilizers of components under this $\sH$-action.
Parts (2) and (3) of Theorem~\ref{thm:Components} emphasize again the importance of $\sH$ for the description of conjugacy classes in $\aH$ (see Remark~\ref{rmk:spherical}). In situations where conjugacy classes in $\sH$ are well-understood, such as when $\sH = \On$, we can thus obtain much information about the geometry of conjugacy classes in $\aH$ by ``lifting" from $\sH$. We give some additional results for the special case $\aH = \aG$ in Proposition~\ref{prop:componentsG}.

\subsection{Coconjugation}

In our final main result, we turn to the question of which $k \in \aH$ conjugate a given $h \in \aH$ to some $h'$ in its conjugacy class. We refer to this question as the \emph{coconjugation problem}. That is, given $h, h' \in \aH$, we will determine the (possibly empty) set $\coconjH{h}{h'}$ of elements $k\in \aH$ such that $khk^{-1}=h'$. As we will show, the solution to the coconjugation problem in $\aH$ crucially involves the fix-sets of elements of $\sH$; that is, the orthogonal complement of the spaces ruling the shape of the conjugacy class itself.  

For any $h' \in \hconj$, the coconjugation set $\coconjH{h}{h'}$ is equal to $k \Cent_{\aH}(h)$ for any $k \in \aH$ such that $khk^{-1} = h'$. 
One could hence say that it is enough to consider centralizers to  fully solve the coconjugation problem. However, in \Cref{thm:coconj} below, we provide an intrinsic description of the coconjugation set that does not require prior knowledge of the centralizer, nor the determination of a conjugating element $k$ as used above. 
Instead, the disjoint union in \Cref{thm:coconj} is parametrized by the following explicitly-defined subset of the coconjugation set $\coconj_{\sH}(h_0, h_0')$.

\begin{definition}[Translation-compatible part of the coconjugation set]
	\label{def:transCompCoConj} Let $\aH = T_\aH \rtimes \sH$ be a split group of Euclidean isometries, let $\lambda, \lambda' \in L_\aH$, and let $h_0, h_0' \in \sH$. 
	The \emph{translation-compatible part} of $\coconj_{\sH}(h_0,h_0')$ is defined by:
	\begin{equation}\label{K0_conj}
		\coconj_{\sH}^{\lambda,\lambda'}(h_0,h_0') = \{ u \in \scoconjH{h_0}{h_0'} \mid \lambda' - u\lambda \in \Mod_\aH(h_0') \}.
	\end{equation}
\end{definition}

\begin{thm}[Coconjugation]
	\label{thm:coconj}
	Let $\aH = T_\aH \rtimes \sH$ be a split group of Euclidean isometries. Let $h=t^{\lambda}h_0$ and $h' = t^{\lambda'}h_0'$ be elements of $\aH$, where $\lambda, \lambda' \in L_\aH$ and $h_0, h_0' \in \sH$.  Then 
	\begin{equation}\label{eq:coconjNonempty}
		\coconjH{h}{h'} \neq \emptyset \;\; \Longleftrightarrow\;  \coconj_{\sH}^{\lambda,\lambda'}(h_0,h_0') \neq \emptyset.
	\end{equation}
	Moreover, if these sets are nonempty, then
	\begin{equation}\label{eq:coconjFix}
	\coconjH{h}{h'} = \bigsqcup_{u \in \coconj_{\sH}^{\lambda,\lambda'}(h_0,h_0')} t^{\eta_{u}+ (\Fix(h_0') \cap L_\aH)} u 
	\end{equation}
	where for each $u$, the element $\eta_{u} \in L_\aH$ is a particular solution to the equation 
	\begin{equation}\label{eq:coconj}
	\lambda'- u\lambda=(\Id - h_0')\eta.
	\end{equation}
	In the special case that $\Fix(h_0)=\{0\}$, we have that \[ \eta_u=(\Id-h_0')^{-1}(\lambda'-u\lambda)\] is the unique solution to~\eqref{eq:coconj}, and $\coconjH{h}{h'}$ is in bijection with $\coconj_{\sH}^{\lambda,\lambda'}(h_0,h_0')$.	
\end{thm}

\noindent Geometrically,~\eqref{eq:coconjFix} means that the coconjugation set $\coconjH{h}{h'}$ lies along translates of the fix-set $\Fix(h_0')$, and so is orthogonal to $\Mov(h_0')$. The reason for this appearance of the fix-set in our description of coconjugation sets is that we are solving Equation~\eqref{eq:coconj}, and $\Fix(h_0') = \Ker(\Id - h_0')$.  In the special case that $h = h'$, Theorem~\ref{thm:coconj} yields a new geometric description of the centralizer $\Cent_{\aH}(h)$.

When nonemptiness of the set $\coconj_{\sH}^{\lambda,\lambda'}(h_0,h_0')$ can be determined, the equivalence~\eqref{eq:coconjNonempty} in Theorem~\ref{thm:coconj} provides an algorithm to solve the conjugation problem in $\aH$. If, in addition, all elements of $\coconj_{\sH}^{\lambda,\lambda'}(h_0,h_0')$ and all solutions to Equation~\eqref{eq:coconj} can be computed, we obtain an algorithm which lists all elements of the coconjugation set.

  \begin{figure}[h]
		\resizebox{0.95\textwidth}{!}
		{\begin{overpic}{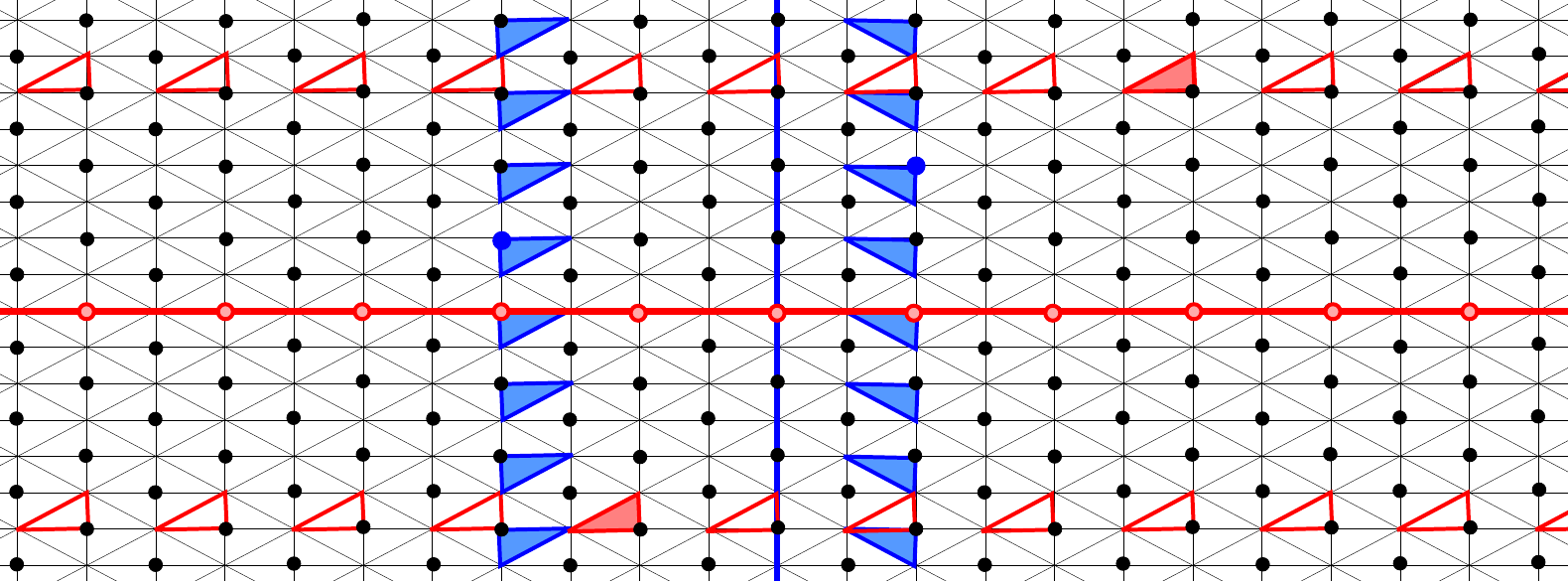}
		\put(50.5,17.6){$\mathbf{1}$}
		\put(47.5,17.6){$\mathbf{s_1}$}
		\put(50.5,16){$\mathbf{s_2}$}
		\put(52,18.5){$\mathbf{\rho}$}
		\put(46.5,16){$\mathbf{s_1 s_2}$}	
		\put(76.5,30){$\mathbf{\lambda}$}
		\put(74.5,31.5){$\mathbf{h}$}
		\put(41.5,2.3){$\mathbf{\lambda'}$}
		\put(39.2,3.7){$\mathbf{h'}$}
		\put(76.5,2.5){$\mathbf{s_2\lambda}$}
		\put(23.5,30.2){$\mathbf{s_1\lambda}$}
		\put(23.5,2.5){$\mathbf{s_1 s_2 \lambda}$}
		\put(30.5,22.7){\textcolor{black}{$\mathbf{\eta_{s_2}}$}}	
		\put(59,25.6){\textcolor{black}{$\mathbf{\eta_{s_1s_2}}$}}	
		\end{overpic}}
		\vspace{-3mm}
		\caption{\footnotesize{The coconjugation set $\coconjH{h}{h'}$, in blue, in the wallpaper group $\mathbf{cmm}$.}}	
		\label{fig:cmm_conj_coconj_reflection}	
\end{figure}
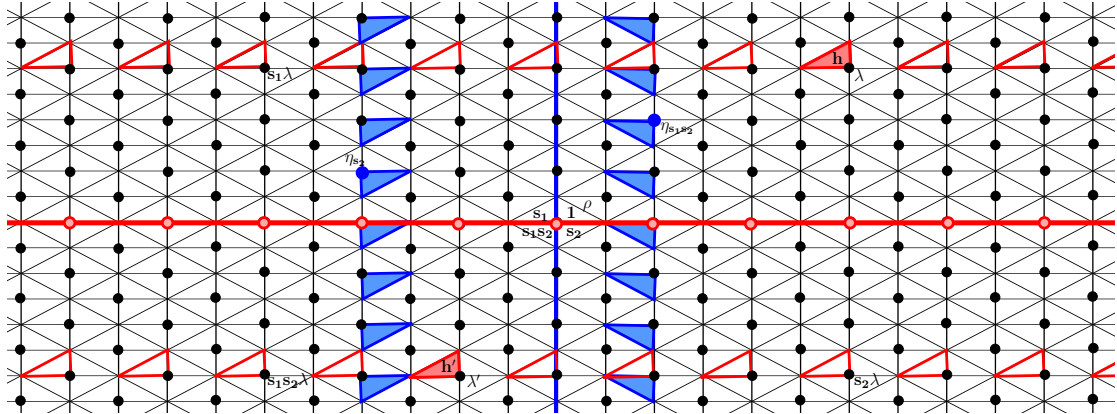

\begin{example}
In $\aH$ the wallpaper group $\mathbf{cmm}$, let $h = t^\lambda s_1$ and $h' = t^{\lambda'}s_1$ be the conjugates shaded red in Figure~\ref{fig:cmm_conj_coconj_reflection} (compare Figure~\ref{fig:cmm_conj_reflection_cropped}). Here, $\Fix(s_1)$ is the vertical blue line and $\Mov(s_1)$ is the horizontal red line. It turns out that $\coconj_{\sH}^{\lambda, \lambda'}( s_1, s_1) = \{ s_2, s_1 s_2 \}$, and we can choose $\eta_{s_2},\eta_{s_1 s_2} \in L_\aH$ as indicated by blue dots. 
\end{example}

\begin{rmk} The proofs in this work use only the semidirect product structure of split subgroups of $\aG$.   We initially proved a version of Theorem~\ref{thm:ConjClass}, investigated the filling property, and identified some elements of coconjugation sets for certain ``standard representatives" in affine Coxeter groups, as part of our study of affine Deligne--Lusztig varieties (see Section 5.1 of~\cite{MST2}).  We then established the results described above for arbitrary elements of affine Coxeter groups, and subsequently realized that our proofs extended immediately to all split~$\aH \leq \aG$. In~\cite{MST5}, we discuss the results of the present work in relation to the literature on Coxeter groups.
\end{rmk}

\begin{rmk} In our companion paper~\cite{MST5}, for certain split crystallographic groups (see Remark~\ref{rmk:cryst}) we refine the relationships between mod-sets and move-sets observed above. We regard the lattice $L_\aH$ as a free $\Z$-module of rank $n$, and prove in~\cite[Theorem 1.6 and Corollary 1.7]{MST5} 
that for any split crystallographic $\aH = T_\aH \rtimes \sH$ which is contained in an affine Coxeter group, and all $h_0 \in \sH$:
\begin{enumerate}
		\item\label{item:rank} the rank of $\ModH(h_0)$ equals the dimension of the move-set $\Mov(h_0)$; 
		\item\label{item:index} $\ModH(h_0)$ is a finite-index submodule of $\Mov(h_0) \cap L_\aH$; and
		\item\label{item:equal} $h_0$ fills its move-set if and only if $L_\aH / \ModH(h_0)$ is torsion-free.	
		\end{enumerate}
Part \eqref{item:rank} here tells us that the move-set can be viewed as the ``enveloping subspace" of the mod-set, while part~\eqref{item:index} implies that the containments in Corollary~\ref{cor:conjMov} are equalities up to finite index. Part \eqref{item:equal} then provides a mechanism (quotients of free $\Z$-modules) for determining which elements fill their move-sets (see Definition~\ref{defn:filling} above).  

 Our proofs of \eqref{item:rank}--\eqref{item:equal} in \cite{MST5} use properties of affine Coxeter groups beyond their semidirect product structure, including their close relationship to finite Weyl groups. We do not know if \eqref{item:rank}--\eqref{item:equal} hold for split crystallographic groups which are not contained in affine Coxeter groups (see Remark~\ref{rmk:affineCoxeter}).
\end{rmk}


\section{Geometry of conjugacy classes}\label{sec:Conj}

In this section, we describe the geometry of conjugacy classes of elements of split subgroups  $\aH = T_\aH \rtimes \sH$ of the isometry group $\aG$ of $\E^n$. We prove  \Cref{thm:ConjClass} in \Cref{sec:geometryConj}. Components and the  effect of conjugation on them are described in \Cref{sec:componentsAction}, where we prove \Cref{thm:Components}. We close the discussion by considering the special case $\aH = \aG$ in \Cref{sec:componentsG}.

\subsection{Geometric description of conjugacy classes}\label{sec:geometryConj}

The mod-set (see Definition~\ref{defn:mod-set}) is the key player in our geometric description of conjugacy classes. As noted after this  definition, we have $\ModH(h) = (\Id - h)L_\aH$. In the remainder of this paper, we will often work with this formulation, as it is the one which arises in our results.

We first record two easy lemmas concerning mod-sets, then use these together with some elementary computations to prove Theorem~\ref{thm:ConjClass}.

\begin{lemma}\label{lem:modConj} For all $h \in \aH$ and all $u \in \sH$, we have $u \ModH(h) = \ModH(u h u^{-1})$.
\end{lemma}
\begin{proof}  Since $L_\aH$ is $\sH$-invariant, we have
	\[
	u(\Id - h)L_\aH = u(\Id - h)u^{-1}L_\aH = (\Id - u h u^{-1})L_\aH,
	\]
which establishes the result.
\end{proof}

\begin{lemma}\label{lem:modShift} For all $\lambda \in L_\aH$ and all $h_0 \in \sH$, we have $\ModH(t^\lambda h_0) = \lambda + \ModH(h_0)$. 
\end{lemma}
\begin{proof}  For any $\mu \in L_\aH$, we have \[ \lambda + (h_0 - \Id)\mu = (\lambda + h_0\mu) - \mu = t^\lambda h_0 \mu -\mu = (t^\lambda h_0 - \Id)\mu.\] The result follows.
\end{proof}

\begin{proof}[Proof of \Cref{thm:ConjClass}]
	Fix $h=t^\lambda h_0$ and let $k = t^\eta u$ be any element of $\aH$, where $\eta \in L_\aH$ and $u \in \sH$.  Compute that
	\begin{eqnarray*}
	khk^{-1} 
	& = & t^\eta u (t^\lambda h_0) u^{-1} t^{-\eta} \\
	& =&  t^\eta (u t^\lambda u^{-1}) (u h_0 u^{-1}) t^{-\eta} \\
	& = & t^\eta t^{u \lambda} \left( (u h_0 u^{-1}) t^{-\eta}  (u h_0 u^{-1})^{-1} \right) u h_0 u^{-1} \\
	& = & t^{\xi} u h_0 u^{-1} 
	\end{eqnarray*}
	where $\xi =  u \lambda + (\Id-u h_0 u^{-1})\eta.$
	Hence $\hconj$ is of the form
	\begin{eqnarray*}
	\hconj 
	& =& \bigcup_{u \in\sH}\{ t^\xi u h_0 u^{-1} \mid \xi \in u \lambda + \ModH(u h_0 u^{-1})  \}.
	\end{eqnarray*}
	Now $\ModH(u h_0 u^{-1}) = u \ModH(h_0)$ by~\Cref{lem:modConj}, while for any $\mu \in \ModH(h_0)$, we have $u \left( t^\mu h \right) u^{-1} = t^{u (\lambda + \mu)} u h_0 u^{-1}$.
	Combining these yields \Cref{eq:conj1}, and then using Lemma~\ref{lem:modShift}, we obtain \Cref{eq:conj2}.  
\end{proof}

We now relate conjugacy classes to filling (see Definition~\ref{defn:filling}). We first observe:

\begin{lemma}\label{lem:modMov} For all $h \in \aH$, we have 
$\ModH(h) \subseteq \Mov(h) \cap L_\aH$.
\end{lemma}
\begin{proof} By definitions, $\ModH(h) \subseteq \Mov(h)$. Now $L_\aH$ is an $\aH$-invariant $\Z$-module, so $(h - \Id)L_\aH \subseteq L_\aH$. The result follows.
\end{proof}

\begin{prop}[Conjugacy classes and filling]\label{prop:conjFilling} For all $h = t^\lambda h_0 \in \aH$, the following are equivalent:
\begin{enumerate}
\item\label{fill:h} $h$ fills its move-set; that is, $\ModH(h) = \Mov(h) \cap L_\aH$;
\item\label{fill:h0} $h_0$ fills its move-set; that is, $\ModH(h_0) = \Mov(h_0) \cap L_\aH$;
\item\label{fill:hconjh0} $\hconj    =  \bigcup_{u \in \sH} u \left( t^{\Mov(h_0) \cap L_\aH} h \right) u^{-1}$; and
\item\label{fill:hconjh} $\hconj =   \bigcup_{u \in \sH} t^{u(\Mov(h) \cap L_\aH)} u h_0  u^{-1}$.
\end{enumerate}
\end{prop}
\begin{proof} We have $\ModH(h) = \lambda + \ModH(h_0)$ by Lemma~\ref{lem:modMov}, $\Mov(h) = \lambda + \Mov(h_0)$ by~\cite[Proposition 1.21]{LMPS}, and $(\lambda + \Mov(h_0)) \cap L_\aH = \lambda + (\Mov(h_0) \cap L_\aH)$. Hence   \eqref{fill:h} and \eqref{fill:h0} are equivalent. Similarly, \eqref{fill:hconjh0} and \eqref{fill:hconjh} are equivalent.  

Theorem~\ref{thm:ConjClass} gives us that \eqref{fill:h0} implies \eqref{fill:hconjh0} (for instance), and we complete the proof by showing that~\eqref{fill:hconjh0} implies~\eqref{fill:h0}. Suppose \eqref{fill:hconjh0} holds, and let $\mu \in \Mov(h_0) \cap L_\aH$. Then there is a $k = t^\eta u \in \aH$, with $\eta \in L_\aH$ and $u \in \sH$, such that $khk^{-1} = u(t^\mu h) u^{-1}$. As in the proof of Theorem~\ref{thm:ConjClass}, we compute that $khk^{-1} = t^\xi uh_0u^{-1}$ where $\xi = u\lambda + (\Id - uh_0 u^{-1})\eta$. On the other hand, $u(t^\mu h) u^{-1} = ut^\mu u^{-1} u t^\lambda h_0 u^{-1} = t^{u\mu + u\lambda} u h_0 u^{-1}$. Hence
\[
u\lambda + (\Id - uh_0u^{-1})\eta = u\lambda + u\mu.
\]
Thus $u \mu$ is an element of $\ModH(uh_0 u^{-1})$. But by Lemma~\ref{lem:modConj}, we have $\ModH(uh_0 u^{-1}) = u \ModH(h_0)$, so $\mu \in \ModH(h_0)$ as required.
\end{proof}

\subsection{Components of conjugacy classes and action by conjugation}\label{sec:componentsAction}

In this section we define components and prove part~(1) of Theorem~\ref{thm:Components}, as well as describing the stabilizers for the action by conjugation on components, in Proposition~\ref{prop:equalcomponents}.

\begin{definition}[Component]\label{defn:component}  Let $h = t^\lambda h_0 \in \aH$, where $\lambda \in L_\aH$ and $h_0 \in \sH$.  The \emph{base component} of $h$ is the subset of $\aH$ given by
\[
\BaseH(h) = t^{\ModH(h_0)} h =  t^{\lambda+\ModH(h_0)} h_0.
\]  
A \emph{component of $\hconj$} is any of the base components of its elements.   
We write $\CompH(h)$ for the set of components of $\hconj$, and $\#\CompH(h)$ for the cardinality of this set. 
\end{definition}

We start by collecting some first properties of components.

\begin{lemma}[Shape of components]
	\label{lem:conjBase}
Let $h = t^\lambda h_0 \in \aH$, where $\lambda \in L_\aH$ and $h_0 \in \sH$. Then: 
	\begin{enumerate}
		\item $\BaseH(h)=t^\lambda\BaseH(h_0)$;
		\item $u \BaseH(h) u^{-1} = \BaseH(u h u^{-1})$ for all $u \in \sH$; and  		
		\item Every component of $\hconj$ is of the form
		\[
		u\BaseH(h)u^{-1} = t^{u\lambda + u\ModH(h_0)} uh_0u^{-1} = t^{u\lambda} \BaseH(uh_0u^{-1})
		\]
		for some $u \in \sH$.
	\end{enumerate}
\end{lemma}
\begin{proof}
	Item (1) is clear from the definition. A straightforward computation using  Lemma~\ref{lem:modConj} implies (2). Combining the first two items with \Cref{thm:ConjClass}, the third item follows. 
\end{proof}

We now consider the effect of conjugating by a translation in $\aH$.

\begin{lemma}
\label{lem:translationsComponents}  
Let $h = t^\lambda h_0 \in \aH$, where $\lambda \in L_\aH$ and $h_0 \in \sH$. 
\begin{enumerate} 
\item\label{lem:conjTrans} For all $\eta \in L_\aH$ and all $k\in \BaseH(h)$ one has 
\[
t^\eta k t^{-\eta} = t^{\eta'} k,
\]
where $\eta' = (\Id - h_0)\eta \in \ModH(h_0)$. In particular, $t^\eta k t^{-\eta} \in \BaseH(h)$.  
	
\item For all $\eta \in L_\aH$ and $u \in \sH$, 
\[
t^\eta \left( u \BaseH(h) u^{-1} \right) t^{-\eta} = u \BaseH(h) u^{-1}.
\]
\item For all $h',h'' \in u \BaseH(h) u^{-1}$, there is an $\eta \in L_\aH$ such that 
\[
t^\eta h' t^{-\eta} = h''.
\]
\end{enumerate}
\end{lemma}
\begin{proof} 
To see (1) let $k=t^{\lambda+\mu}h_0$ for $\mu \in \ModH(h_0)$ and  compute $ t^\eta \left( t^{\lambda + \mu} h_0 \right) t^{-\eta} = t^{\lambda + \mu} t^\eta t^{- h_0 \eta} h_0 = t^{\lambda + \mu + (\Id - h_0)\eta}h_0 = t^{(\Id - h_0)\eta} k$, 
as required.	
For (2), by \Cref{lem:conjBase} it suffices to consider $t^\eta \BaseH(h) t^{-\eta}$.  The result then follows from (1).  

As a preparation for (3) observe the following: an arbitrary element $h'$ of the component $u \BaseH(h) u^{-1}$ has the form $h' = t^{u(\lambda + \mu')}u h_0 u^{-1}$, where $\mu' \in \ModH(h_0)$.  Then a similar computation to that in the proof of part~\eqref{lem:conjTrans} establishes the following equation: 
for all $\eta \in L_\aH$ and all $h' \in  \BaseH(uhu^{-1})$, we have
\begin{equation}
	\label{eq:interim}
	t^\eta h' t^{-\eta} = t^{\eta'} h'
\end{equation}
where $\eta' = (\Id - u h_0 u^{-1})\eta \in \ModH(u h_0 u^{-1})$.  In particular, $t^\eta h' t^{-\eta}  \in \BaseH(uhu^{-1})$.

For (3), since $L_\aH$ is $\sH$-invariant, it suffices to consider the base component.  So let $h' = t^{\lambda + \mu'} h_0 $ and $h'' = t^{\lambda + \mu''} h_0 $ be elements of $\BaseH(h_0)$, where $\mu',\mu'' \in \ModH(h_0)$.   Then 
\[
h'' (h')^{-1} = t^{\lambda + \mu''} h_0 h_0^{-1} t^{-(\lambda +\mu')} = t^{\mu'' - \mu'}
\]
and so putting $\eta' = \mu'' - \mu'$, we have $h'' = t^{\eta'} h'$.  But then $\eta' \in \ModH(h_0)$, and so $\eta' = (\Id - h_0)\eta$ for some $\eta \in L_\aH$.  Combining this with \Cref{eq:interim} completes the proof.
\end{proof}

The next result establishes Theorem~\ref{thm:Components}(1).

\begin{corollary}[Conjugation action]
	\label{prop:ConjAction} 
Let $h = t^\lambda h_0 \in \aH$, where $\lambda \in L_\aH$ and $h_0 \in \sH$.
	\begin{enumerate}
		\item The conjugation action of $\sH$ on $\aH$ induces a transitive action of $\sH$ on the set $\CompH(h)$. 
		\item The conjugation action of $T_\aH$ on $\aH$ stabilizes each component of $\hconj$ setwise, and acts transitively when restricted to any of the components. 
	\end{enumerate}
\end{corollary}
\begin{proof}
	It is immediate from \Cref{thm:ConjClass} and \Cref{defn:component} that the conjugacy class $\hconj$ is the union of  its components, and that $\sH$ acts transitively by conjugation on the elements of $\CompH(h)$. This establishes (1).   
	Item (2) follows from Lemma~\ref{lem:translationsComponents}. 
\end{proof}

We close this section by describing the stabilizers of components under the conjugation action of $\sH$ on $\CompH(h)$. An alternative point of view here is that we are attempting to describe the conjugacy class $\hconj$ as a disjoint union of components. 
Recall from \Cref{thm:ConjClass} and \Cref{defn:component} that 
\begin{equation}
	\hconj    
	= \bigcup_{u \in \sH} u \BaseH(h) u^{-1}. 
\end{equation}
So in order to write this union as a disjoint union we need to understand which conjugates of the base component are equal. It suffices, as done in the next proposition, to find conditions that characterize the case where $u \BaseH(h) u^{-1} = \BaseH(h) $. That is, we will determine the stabilizer of $\BaseH(h)$ in $\sH$.

\begin{prop}
	\label{prop:equalcomponents}
	Let $h = t^\lambda h_0 \in \aH$, where $\lambda \in L_\aH$ and $h_0 \in \sH$.  Then for all $u \in \sH$, 
	\[
		u \BaseH(h) u^{-1} =  \BaseH(h)
 		\Longleftrightarrow 
		\begin{cases}
			\text{(1) } & u\in\Cent_{\sH}(h_0), \;\text{ and }\\
 			\text{(2) } & (\Id-u)\lambda \in \Mod_\aH(h_0).
		\end{cases} \\
	\]
\end{prop}
\begin{proof}
We compute: 
\[
u \BaseH(h) u^{-1} = t^{u\lambda} u\BaseH(h_0)u^{-1} = t^{u\lambda} \BaseH(u h_0 u^{-1}). 
\]	
Hence $u \BaseH(h) u^{-1}=\BaseH(h)$ is equivalent to 
$
t^{u\lambda} \BaseH(u h_0 u^{-1})= t^\lambda\BaseH(h_0)
$
and hence to
\[
t^{u\lambda+\ModH(uh_0 u^{-1})} u h_0 u^{-1} = t^{\lambda+\ModH(h_0)} h_0.
\]
This equation holds if and only if both $u h_0 u^{-1}=h_0$ (and hence (1) is satisfied) and the following equation holds: 
\begin{equation}\label{eqn:modsetsequal}
u\lambda+\ModH(uh_0 u^{-1})=\lambda+\ModH(h_0). 
\end{equation}
Under condition (1), Equation~\eqref{eqn:modsetsequal} is equivalent to (2). 
\end{proof}

We note that condition (2) in \Cref{prop:equalcomponents} implies that the mod-sets of $u$ and $h_0$ intersect, and that if $u$ fixes $\lambda$ then (2) holds. It seems difficult to make any further general observations on this condition.

The next result completes the proof of Theorem~\ref{thm:Components}.

\begin{lemma}\label{lem:componentsInequality} 
	Let $h = t^\lambda h_0 \in \aH$, where $\lambda \in L_\aH$ and $h_0 \in \sH$. 
	\begin{enumerate}
		\item The linearization map sending $\lambda + \ModH(h_0)$ to $\ModH(h_0)$ induces a surjection
		\[
		\Comp\hconj \twoheadrightarrow \Comp[h_0]_\aH. 
		\]
		Hence $\#\Comp\hconj \geq \#\Comp[h_0]_\aH.$ 
		\item There is a natural bijection between the components of $[h_0]_\aH$ and the elements of the spherical conjugacy class $[h_0]_{\sH}$. Hence \[\#\Comp[h_0]_\aH = \# \{ u h_0 u^{-1} \mid u \in \sH \}.\] 
	\end{enumerate}  
\end{lemma}
\begin{proof} Part (1) follows from the definition of components (see Definition~\ref{defn:component}), and part~(2) from this definition and Lemma~\ref{lem:conjBase} with $h = h_0$.
\end{proof}

Figure~\ref{fig:cmm_cuspidal} shows that the inequality in \Cref{lem:componentsInequality}(1) can be either strict or an equality, depending upon some fine behavior.
It would be interesting to see whether one can characterize those elements for which  equality holds.

\subsection{Components in $\aG$}\label{sec:componentsG}

We now consider components in the special case $\aH = \aG$.

\begin{prop}\label{prop:componentsG}  Let $g = t^\lambda g_0 \in \aG$, where $\lambda \in \R^n$ and $g_0 \in \On$. 
\begin{enumerate} 
\item If $\Mov(g_0) = \R^n$, then $[g]_\aG = [g_0]_\aG$ and the components of this conjugacy class are the sets of the form
\[
\{ t^\mu u g_0 u^{-1} \mid \mu \in \R^n, u \in \On \}. 
\]
\item  If $\Mov(g_0)$ is $m$-dimensional with $0 < m < n$, and $\lambda \not \in \Mov(g_0)$, then the fibers of the  surjection
\[
\CompG(g) \twoheadrightarrow \CompG(g_0)
\]
induced by linearization are all of cardinality $\geq 2$.  Moreover, in this case the sets $\CompG(g)$ and $\CompG(g_0)$ both have cardinality $\aleph_1$.
\end{enumerate}
\end{prop}
\begin{proof} If $\Mov(g_0) = \R^n$ then we have $u(\lambda + \Mov(g_0)) = \R^n$ as well, for any $u \in \On$ and any $\lambda \in \R^n$, including $\lambda = 0$.  Hence $[g]_\aG = [g_0]_\aG$ has components as stated in (1). 

In case (2), we have that $\lambda + \Mov(g_0)$ is an $m$-dimensional affine subspace which is distinct from $\Mov(g_0)$.  Now $\On$ contains the isometry $-\Id$, while $-\lambda \neq \lambda$ since $\lambda \neq 0$.  Hence $\lambda + \Mov(g_0) \neq (-\Id)(\lambda + \Mov(g_0))$, and it follows that the components $\BaseG(g)$ and $(-\Id)\BaseG(g)(-\Id)^{-1}$ are distinct (even though the elements of these components have the same spherical part, since $-\Id$ is central in~$\On$).  Hence conjugating these two components by any element of $\sH$ also results in two distinct components.  As linearization sends both $\BaseG(g)$ and $(-\Id)\BaseG(g)(-\Id)^{-1}$ to $\BaseG(g_0)$, the map $\CompG(g) \to \CompG(g_0)$ induced by linearization has fibers of cardinality $\geq 2$.  

For the final claim in (2), by Lemma~\ref{lem:componentsInequality} it suffices to see that $g_0$ has $\aleph_1$-many distinct conjugates in $\On$.  Since $0 < m < n$, we have that $\On$ acts transitively on the set of $m$-dimensional subspaces of $\E^n$, which has cardinality $\aleph_1$.  On the other hand, $u \Mov(g_0) = \Mov(u g_0 u^{-1})$ by Lemma~\ref{lem:modConj}.  The result follows.
\end{proof}

\begin{rmk}\label{eg:ConjClassGcuspidal}  If $\Mov(g_0) = \R^n$ then by \Cref{prop:componentsG}(1), for any $\lambda \in \R^n$ we can view each component of $[t^\lambda g_0]_\aG = [g_0]_\aG$ as a ``sheet" of $\R^n$ with the corresponding $\On$-conjugate of $g_0$ sitting at each point, and the conjugation action of $\On$ permuting these sheets.  Alternatively, we can view the entire conjugacy class $[t^\lambda g_0]_\aG$ as a single copy of~$\R^n$ with the entire $\On$-conjugacy class of~$g_0$ sitting at each point. 
\end{rmk}

\begin{example}  We generalize the first paragraph of Example~\ref{eg:introG}. Let $r \in \On$ be a reflection, so that $\Fix(r)$ is a linear hyperplane, and write $p_r : \R^n \to \Fix(r)$ for the orthogonal projection onto this hyperplane. Then for all $\lambda \in \R^n$ not in the line through the origin $\Mov(r)$, the components of the  class $[t^\lambda r]_\aG$ are the lines tangent to the sphere in $\R^n$ of radius $\| p_r(\lambda) \| > 0$.  The map $\CompG(t^\lambda r) \twoheadrightarrow \CompG(r)$ from \Cref{prop:componentsG}(2) has fibers in bijection with the points of a sphere of dimension $n-2$, so this surjection is $2$-to-$1$ when $n = 2$ and $\aleph_1$-to-$1$ for all $n \geq 3$.
\end{example}


\section{Geometry of coconjugation sets}\label{sec:CoConj}

We now prove \Cref{thm:coconj}. Some key observations are gathered in the next result. 

\begin{prop}
	\label{prop:coconj}
	Let $h=t^{\lambda}h_0$ and $h'=t^{\lambda'}h_0'$ be elements of $\aH$, where $\lambda, \lambda' \in L_\aH$ and $h_0,h_0' \in \sH$.
	Then for all $u \in \sH$:
	\begin{enumerate}
	\item\label{i:1co} For any $\eta \in L_\aH$,
		\[
		k=t^\eta u\in \coconjH{h}{h'}  \Longleftrightarrow u\in \scoconjH{h_0}{h_0'} \text{ and } \lambda'- u\lambda = (\Id-h_0')\eta.
		\]
	\item\label{i:2co} There exists at least one $\eta \in L_\aH$ such that  $\lambda'- u\lambda = (\Id-h_0')\eta$ if and only if $\lambda' - u \lambda \in \Mod_\aH(h_0')$.
	\item\label{i:3co} Let $\eta_0 \in  L_\aH$ and suppose that $t^{\eta_0}u \in \coconjH{h}{h'}$. Then for all $\eta \in L_\aH$, we have 
		\[
		k = t^\eta u \in \coconjH{h}{h'} \Longleftrightarrow \eta\in \eta_0+(\Fix (h_0') \cap L_\aH).
		\] 	
	\end{enumerate}
\end{prop}
\begin{proof} Let $\eta$ and $k$ be as in \eqref{i:1co}. By the same computation as in the proof of \Cref{thm:ConjClass}, we have 
$
khk^{-1} = t^{u \lambda + (\Id - u h_0 u^{-1})\eta} u h_0 u^{-1}.
$
Thus $k \in \coconjH{h}{h'}$ if and only if $u \in \coconj_{\sH}(h_0,h_0')$ and $\lambda' = u \lambda + (\Id - h_0')\eta$.  Part~\eqref{i:1co} follows.

Part~\eqref{i:2co} is immediate from the definition $\Mod_\aH(h_0') = (\Id - h_0')L_\aH$.

To prove item \eqref{i:3co} observe that by part \eqref{i:1co} we have 
\begin{equation}
	\label{eq:proof41}
	k = t^\eta u \in \coconjH{h}{h'}\;\Longleftrightarrow\;
	\lambda' - u \lambda = (\Id - h'_0)\eta.
\end{equation} 

By our assumption we have that $\lambda'-u\lambda=(\Id-h'_0)\eta_0$. We subtract this formula from \Cref{eq:proof41} and obtain
\[
0 = (\Id-u)\lambda - (\Id-u)\lambda = (\Id-h'_0)\eta - (\Id-h'_0)\eta_0,
\]
which is equivalent to the fact that $(\eta - \eta_{0}) \in \Ker(\Id - h'_0) = \Fix(h_0)$.  Since $\eta - \eta_{0} \in L_\aH$, we obtain item~\eqref{i:3co}.
\end{proof}	

Recall from Definition~\ref{def:transCompCoConj} that the translation-compatible part of the coconjugation set is given by:
\[
\coconj_{\sH}^{\lambda,\lambda'}(h_0,h_0') = \{ u \in \scoconjH{h_0}{h_0'} \mid \lambda' - u\lambda \in \Mod_\aH(h_0') \}.
\]
This definition is motivated by Proposition~\ref{prop:coconj}.

\begin{proof}[Proof of \Cref{thm:coconj}] 
	The equivalence~\eqref{eq:coconjNonempty} is immediate from \Cref{prop:coconj} and the definition of the set $\coconj_{\sH}^{\lambda,\lambda'}(h_0,h_0')$.  
	
	We may henceforth assume that $\coconj_{\sH}^{\lambda,\lambda'}(h_0,h_0') \neq \emptyset$. 
	Then, since $h_0$ and $h_0'$ are conjugate in $\sH$, we have $\Fix(h_0) = \{ 0 \}$ if and only if $\Fix(h_0') = \{ 0 \}$.   We obtain the given shape of $\coconjH{h}{h'}$ as a consequence of \Cref{prop:coconj}. Combining this with the fact that $\Fix(h_0')=\{0\}$ is equivalent to $(\Id-h_0')$ being invertible implies the statement about the case of the fixed point set being empty. 
\end{proof}

\section{Algorithmic solution to the (co)conjugacy problem}\label{sec:algorithm}

We now use the results of \Cref{sec:CoConj} to sketch an algorithm to solve the conjugacy problem and to compute all coconjugation sets. We restrict to $\aH$ split crystallographic, so that $\sH$ is finite and the lattice $L_\aH$ may be regarded a free $\Z$-module of rank $n$ (see Remark~\ref{rmk:cryst}).

\begin{algorithm}
	Let $h=t^{\lambda} h_0$ and $h'=t^{\lambda'} h_0'$ be elements of $\aH$, where $\lambda, \lambda' \in L_\aH$ and $h_0, h_0' \in \sH$.  We want to determine $\coconjH{h}{h'}$. 
	
	\begin{enumerate}
		\item Determine whether $h$ and $h'$ are conjugate:
		\begin{enumerate}
			\item\label{i:scoconj} If $\scoconjH{h_0}{h_0'}=\emptyset$, then $\coconjH{h}{h'}=\emptyset$. 
			\item\label{i:scoconjEq} If $\lambda' - k_0\lambda \notin \Mod_\aH(h_0')$ for all $k_0\in \scoconjH{h_0}{h_0'}$, then $\coconjH{h}{h'}=\emptyset$.  
		\end{enumerate}
		\item\label{i:coconj} If neither of the above two cases appeared, then $h$ and $h'$ are conjugate (by \Cref{prop:coconj}).  In this case $\coconjH{h}{h'}$ is obtained as in \Cref{thm:coconj}.  
	\end{enumerate}
\end{algorithm}

We have not implemented this algorithm, nor do we know its complexity, and we expect that addressing these questions would be a substantial endeavor. Part of the algorithm involves just the finite group $H_0$, and so could be done by brute force if necessary. However a serious implementation requires an efficient solution to the conjugacy problem in $H_0$, followed by computation of the spherical coconjugation set $\scoconjH{h_0}{h_0'}$ (if it is known to be nonempty).  Even for $H_0$ a finite Weyl group, carrying out these steps would likely involve similar efforts to those recently undertaken for finite groups of Lie type in the monograph~\cite{DFLOB}. We then need to know whether, for some $u\in \scoconjH{h_0}{h_0'}$, the $\Z$-linear equation $\lambda' - u \lambda = (\Id - h_0')\eta$ has a solution in $L_\aH$. That is, we need to determine whether certain $\Z$-linear equations have any integral solution. Finding the entire coconjugation set $\coconjH{h}{h'}$ then requires finding all integral solutions to these $\Z$-linear equations. There are various algorithms for such questions, which will require separate analysis to the work required for just $\sH$.

\renewcommand{\refname}{Bibliography}
\bibliography{bibliographyConj}

\begin{thebibliography}{LMPS19}

\bibitem[BM15]{BradyMcCammond}
Noel Brady and Jon McCammond.
\newblock Factoring {E}uclidean isometries.
\newblock {\em Internat. J. Algebra Comput.}, 25(1-2):325--347, 2015.

\bibitem[BtD85]{BroeckerDieck}
Theodor Br\"{o}cker and Tammo tom Dieck.
\newblock {\em Representations of compact {L}ie groups}, volume~98 of {\em
  Graduate Texts in Mathematics}.
\newblock Springer-Verlag, New York, 1985.

\bibitem[Dav08]{Davis}
Michael~W. Davis.
\newblock {\em The geometry and topology of {C}oxeter groups}, volume~32 of
  {\em London Mathematical Society Monographs Series}.
\newblock Princeton University Press, Princeton, NJ, 2008.

\bibitem[FLO24]{DFLOB}
Giovanni~De Franceschi, Martin~W. Liebeck, and E.~A. O'Brien.
\newblock Conjugacy in finite classical groups, 2024.
\newblock \href {http://arxiv.org/abs/2401.07557} {\path{arXiv:2401.07557}}.

\bibitem[LMPS19]{LMPS}
Joel~Brewster Lewis, Jon McCammond, T.~Kyle Petersen, and Petra Schwer.
\newblock Computing reflection length in an affine {C}oxeter group.
\newblock {\em Trans. Amer. Math. Soc.}, 371(6):4097--4127, 2019.

\bibitem[McC18]{McCammond}
Jon McCammond.
\newblock The structure of {e}uclidean {A}rtin groups.
\newblock In {\em Geometric and cohomological group theory}, volume 444 of {\em
  London Math. Soc. Lecture Note Ser.}, pages 82--114. Cambridge Univ. Press,
  Cambridge, 2018.

\bibitem[MS17]{McCammondSulway}
Jon McCammond and Robert Sulway.
\newblock Artin groups of {E}uclidean type.
\newblock {\em Invent. Math.}, 210(1):231--282, 2017.

\bibitem[MST23]{MST2}
Elizabeth Mili{\'c}evi{\'c}, Petra Schwer, and Anne Thomas.
\newblock Affine {Deligne}-{Lusztig} varieties and folded galleries governed by
  chimneys.
\newblock {\em Ann. Inst. Fourier}, 73(6):2469--2541, 2023.

\bibitem[MST24]{MST5}
Elizabeth Mili{\'c}evi{\'c}, Petra Schwer, and Anne Thomas.
\newblock The geometry of conjugation in affine {C}oxeter groups, 2024.
\newblock \href {http://arxiv.org/abs/2407.08080v2}
  {\path{arXiv:2407.08080v2}}.

\bibitem[Rat19]{Ratcliffe}
John~G. Ratcliffe.
\newblock {\em Foundations of hyperbolic manifolds}, volume 149 of {\em
  Graduate Texts in Mathematics}.
\newblock Springer, Cham, third edition, 2019.

\end{thebibliography}
\bibliographystyle{alphaurl}

\typeout{get arXiv to do 4 passes: Label(s) may have changed. Rerun} 

\end{document}